\documentclass[11pt]{article}
\usepackage{amssymb}
\usepackage{amsmath}
\allowdisplaybreaks
\usepackage{amsthm}
\allowdisplaybreaks
\usepackage{amscd}
\usepackage{graphicx}
\usepackage{ifpdf}
\ifpdf
  \usepackage[colorlinks=true,linkcolor=blue,citecolor=red, final,hyperindex]{hyperref}
\else
  \usepackage[colorlinks,final,backref=page ,hyperindex,hypertex]{hyperref}
\fi
\usepackage{footnote}
\usepackage{fancyhdr}

\allowdisplaybreaks
\renewcommand{\theequation}{\arabic{section}.\arabic{equation}}
\def\<{\langle}
\def\>{\rangle}
\def\a{\alpha}
\def\b{\beta}

\def\c{\cdot}

\def\d{\delta}

\def\f{\phi}
\def\p{\psi}

\newtheorem{thm}{Theorem}[section]
\newtheorem{cor}{Corollary}[section]
\newtheorem{rem}{Remark}[section]

\newtheorem{prop}{Proposition}[section]
\newtheorem{lem}{Lemma}[section]

\theoremstyle{definition}
\newtheorem{df}{Definition}[section]

\theoremstyle{definition}
\newtheorem{exa}{Example}[section]

\date{}
\begin{document}
\title{\sf Deformations and Extensions  of BiHom-alternative algebras}
\author{T. Chtioui, S. Mabrouk, A. Makhlouf}
\author{{ Taoufik Chtioui$^{1}$
 \footnote{ E-mail: chtioui.taoufik@yahoo.fr}
, Sami Mabrouk$^{2}$
 \footnote{  E-mail: mabrouksami00@yahoo.fr}
, Abdenacer Makhlouf$^{3}$
 \footnote{Corresponding author,  E-mail: Abdenacer.Makhlouf@uha.fr}
}\\
{\small 1.  University of Sfax, Faculty of Sciences Sfax,  BP
1171, 3038 Sfax, Tunisia} \\
{\small 2.  University of Gafsa, Faculty of Sciences Gafsa, 2112 Gafsa, Tunisia}\\
{\small 3.~ IRIMAS - Département de Mathématiques, 18, rue des frères Lumière,
F-68093 Mulhouse, France}}
 \maketitle

\begin{abstract}
 The aim of this paper is to deal with  BiHom-alternative algebras which are  a generalization of  alternative and Hom-alternative algebras, their structure  is  defined with two commuting multiplicative linear maps. We study  cohomology and one-parameter formal deformation theory of left BiHom-alternative algebras. Moreover, we study central and $T_\theta$-extensions of  BiHom-alternative algebras and their relationship  with cohomology. 
Finally, we  investigate generalized derivations  and give some relevant results.
\end{abstract}


\textbf{M.S.C}: 17D05, 17D15, 14F10,  06B15,  53D55.

\tableofcontents

\section*{Introduction}
\ \ \
Algebras of Hom-type have been recently investigated by many authors. They appeared in the physics literature, in the ninetieth, in quantum deformations of some algebras of vector fields, such as the Witt and Virasoro algebras.  Hartwig, Larsson and Silvestrov introduced and studied, in \cite{hls}, the classes of quasi-Lie, quasi-Hom-Lie and Hom-Lie algebras.
The Hom-analogue of associative, alternative, Jordan  and Novikov algebras were introduced in \cite{ms,mak,yau}.
A categorical approach to Hom-algebras was first  investigated in  \cite{cg}
 and a generalization has been introduced  in \cite{gmmp}. A construction of a Hom-category including a group action brought to concepts of BiHom-type algebras. Hence BiHom-associative and BiHom-Lie algebras were introduced in \cite{gmmp}.
 In \cite{lmmp}, BiHom-pre-Lie algebras and BiHom-Leibniz algebras were defined and studied  via Rota-Baxter operators. The BiHom-versions of alternative, pre-alternative, quadri-alternative and Malcev algebras were introduced in \cite{ChtiouiMabroukMakhlouf} and \cite{BiHomPreAlt}. The authors gave several constructions connecting such structures.
 
Deformation theory arose essentially from geometry and physics.  It is a tool to study a mathematical object by deforming it into a family of the same kind of objects depending on a certain parameter. In the sixties, M. Gerstenhaber introduced algebraic formal deformations for associative algebras in a series of papers (see \cite{Gerstenhaber1,Gerstenhaber2,Gerstenhaber3,Gerstenhaber4}). He used formal power series and showed that the theory is intimately connected to the cohomology of the algebra. The same approach was extended to several algebraic structures (see \cite{am1,am2,cce1,cce2,mak1,mak2}). Other approaches to study deformations exist. For details see \cite{gm1,gm2,gr,mst,f}.

In this paper, we introduce representations, cohomology and formal deformation theory of left BiHom-alternative algebras. This concept generalize the classical and Hom-type cases studied by M. Elhamdadi and A. Makhlouf in \cite{em1, em2}.  In Section 1, we review the basic definitions and properties related to left BiHom-alternative algebras and recall  their representations.
In section 2,  we define first, second and third coboundary operators and corresponding cohomology groups of left BiHom-alternative algebras.
Section 3 is dedicated to develop a one-parameter formal deformation theory for left BiHom-alternative algebras.
Section 4   is devoted to the study of central and $T_\theta$-extensions  of left BiHom-alternative algebras.
In the lat section, we provide some results involving generalized derivations on BiHom-alternative algebras.

Throughout this paper $\mathbb{K}$ is a field of characteristic $0$ and all vector spaces are over $\mathbb{K}$.
We refer to a  BiHom-algebra  as quadruple $(A,\mu,\alpha,\beta)$ where $A$ is a vector space, $\mu$ is a multiplication and $\alpha,\beta$ are two linear maps. It is said to be regular if $\alpha,\beta$ are invertible. A BiHom-associator is a trilinear map $as_{\alpha,\beta}$ defined for all $x,y,z\in A$ by
$as_{\alpha,\beta}(x,y,z)=(xy)\beta(z)-\alpha(x) (yz)$.
When there is no ambiguity, we denote for simplicity the multiplication and composition by concatenation.

\renewcommand{\theequation}{\thesection.\arabic{equation}}
\section{Preliminaries and Basics}
In this section, we give some basics and properties on left BiHom-alternative algebras. We refer for the definitions to \cite{ChtiouiMabroukMakhlouf}.

\begin{df}
A left BiHom-alternative algebra $($resp. right BiHom-alternative algebra) is a quadruple $( A,\mu,\alpha,\beta)$ consisting of  a $\mathbb{K}$-vector space $ A$, a  bilinear map $\mu: A \times  A \longrightarrow  A$ and two homomorphisms $\alpha,\beta:  A \longrightarrow A$ such that $\alpha \beta=\beta\alpha$, $\alpha\mu=\mu\alpha^{\otimes^2}$ and $\beta\mu=\mu\beta^{\otimes^2}$ that satisfy the left BiHom-alternative identity, i.e. for all $x,y,z \in  A$, one has
\begin{equation}\label{sa1}
as_{\alpha,\beta}(\beta(x),\alpha(y),z)+
as_{\alpha,\beta}(\beta(y),\alpha(x),z)=0,
\end{equation}
respectively, the right BiHom-alternative identity, i.e. for all $x,y,z \in  A$, one has
\begin{equation}\label{sa2}
as_{\alpha,\beta}(x,\beta(y),\alpha(z))+
as_{\alpha,\beta}(x,\beta(z),\alpha(y))=0.
\end{equation}
A BiHom-alternative algebra is one which is both a left and a right BiHom-alternative algebra.
\end{df}

Observe that when $\alpha=\beta=Id$ the left BiHom-alternative identity $(\ref{sa1})$ reduces to the  left alternative identity.

\begin{lem}A quadruple   $(A,\mu,\alpha,\beta)$ is a left BiHom-alternative algebra $($resp. right BiHom-alternative algebra) if and only if
$$as_{\alpha,\beta}(\beta(x),\alpha(x),y)=0$$
respectively
$$as_{\alpha,\beta}(x,\beta(y),\alpha(y))=0$$
for any $x,y \in A$. (since the characteristic of the field $\mathbb{K}$ is not  $2$).
\end{lem}

\begin{df}
Let $(A,\mu,\alpha,\beta)$ and $(A^{'},\mu^{'},\alpha^{'},\beta^{'})$ be two  BiHom-alternative algebras. A  homomorphism $f:A\longrightarrow A^{'}$ is said to be a morphism of  BiHom-alternative algebra if the following hold
$$f \circ \mu= \mu^{'}\circ (f \otimes f),~~f\circ \alpha=\alpha^{'}\circ f~~\textrm{and}~~f\circ \beta=\beta^{'}\circ f. $$
\end{df}

\begin{df}
Let $( A,\mu,\a,\b)$ be a BiHom-alternative algebra and $V$ be a vector space. Let
$\ell,r:  A \to gl(V)$ and  $\phi,\psi \in gl(V)$  two commuting linear maps. Then $(V,\ell,r,\phi,\psi)$ is
called a bimodule, or a representation, of $( A,\mu,\a,\b)$, if for any $x,y \in A,\ v \in V$,
\begin{align}
&\phi \ell(x)=\ell(\a(x))\phi,\  \phi r(x)=r(\a(x))\phi,\\
& \psi \ell(x)=\ell(\b(x))\psi,\  \psi r(x)=r(\b(x))\psi \\
& \ell(\b(x)\a(x))\psi(v)  = \ell(\a\b(x))\ell(\a(x))v ,  \label{rep1}\\
& r(\b(x)\a(x))\phi(v)=r(\a\b(x))r(\b(x))v, \label{rep2}\\
&r(\b(y))\ell(\b(x))\phi(v)-\ell(\a\b(x))r(y)\phi(v) = r(\a(x)y)\phi\psi(v)-r(\b(y))r(\a(x))\psi(v) ,\label{rep3}\\
&\ell(\a(y))r(\a(x))\psi(v)-r(\a\b(x))\ell(y)\psi(v)=\ell(y\b(x))\phi\psi(v)-\ell(\a(y))\ell(\b(x))\phi(v).\label{rep4}
\end{align}
\end{df}
The following  observation is straightforward. 
\begin{prop}A tuple 
$(V,\ell,r,\phi,\psi)$ is a bimodule of a BiHom-alternative algebra $( A,\mu,\a,\b)$ if and only if the direct sum $( A\oplus V,\ast,\a+\phi,\b+\psi)$  is turned into a BiHom-alternative algebra (the semidirect product) where
\begin{align}
   & (x_1+v_1)\ast (x_2+v_2)=\mu(x_1,x_2)+\ell(x_1)v_2+r(x_2)v_1, \label{product}\\
   & (\a+\phi)(x+v)=\a(x)+\phi(v),\ \ (\b+\psi)(x+v)=\b(x)+\psi(v).
    \end{align}
\end{prop}

\begin{rem}
  Let $( A,\mu,\a,\b)$ be a left BiHom-alternative algebra, define the operator multiplication $L,R:A\to gl(A)$ by $L(x)(y)=R(y)(x)=xy$ for $x,\ y$ in $
  A$. Then $(A,L,R,\a,\b)$ is a representation of $A$ called the adjoint representation.
\end{rem}

 \section{Cohomology of left BiHom-alternative algebras}

In this section, we
 construct a cochain complex that defines a BiHom-type  cohomology
of a left BiHom-alternative algebra in a $A$-module $V$.

Let $(V,\ell,r,\f,\p)$ be a representation of  a left BiHom-alternative algebras $(A,\mu, \alpha,\beta)$.
\begin{df}
 A  left BiHom-alternative $n$-cochain on $A$ with values in $V$  is a $n$-linear map $f$ from $A\times\cdots\times A$ ($n$-times) to $V$ satisfying :
\begin{eqnarray}
  \f(f(x_1,\dots,x_n)) &=&f(\alpha(x_1),\dots,\alpha(x_n)), \label{Bihom-cochain1}\\
 \p(f(x_1,\dots,x_n)) &=&f(\beta(x_1),\dots,\beta(x_n)),\label{Bihom-cochain2}
\end{eqnarray}

\end{df}
The set of left BiHom-alternative $n$-cochains on $A$ with values in $V$ is denoted by $C^n_{\a,\b,\f,\p}(A;V)$.

For $n=0$, we have $C^0_{\a,\b,\f,\p}(A;V)=V$.
\begin{df}
The coboundary operator $\delta^1: C^1_{\a,\b}(A;V)  \rightarrow C^2_{\a,\b}(A;V)$ is defined by:
\begin{equation}
 \delta^1f(x,y)= \ell(x)f(y) +r(y)f(x) -f(xy).
\end{equation}
\end{df}
We remark that the first differential of a left BiHom-alternative algebra is similar to the first differential map of Hochschild cohomology of a BiHom-associative algebra (1-cocycles are derivations) given in \cite{ApurbaDas}.
\begin{df}
The coboundary operator $\delta^2: C^2_{\a,\b}(A;V)  \rightarrow C^3_{\a,\b}(A;V)$ is defined as follow:
\begin{eqnarray*}
   \delta^2f(x,y,z)&=& r(\beta(z))f(\beta(x),\alpha(y))-
   \ell(\alpha\beta(x))f(\alpha(y),z) \\
   &+&r(\beta(z))f(\beta(y),\alpha(x))-
   \ell(\alpha\beta(y))f(\alpha(x),z)  \\
   &+&f(\beta(x)\alpha(y),\beta(z))-
   f(\alpha\beta(x),\alpha(y)z) \\
   &+&  f(\beta(y)\alpha(x),\beta(z))-
   f(\alpha\beta(y),\alpha(x)z).
\end{eqnarray*}
\end{df}

\begin{rem}
We have $\delta^2f=\delta^2f \circ \tau_{12}$, where $\tau_{12}(x,y,z)=(y,x,z),$ for all $x,y,z \in A$.
\end{rem}

\begin{lem} The operators $\delta_1$ and $\delta_2$ are well defined, that is
\begin{align*}
    & \delta^1 f \circ \a^{\otimes2}= \f  \circ \delta^1 f, \quad  \delta^1 f \circ \b^{\otimes2}= \p  \circ \delta^1 f, \\
      & \delta^2 f \circ \a^{\otimes3}= \f  \circ \delta^2 f, \quad  \delta^2 f \circ \b^{\otimes3}= \p  \circ \delta^2 f.
\end{align*}
\begin{proof}
For any $x,y,z \in A$, we have
\begin{align*}
      \d^1f \a(x,y)&=\d^1f(\a(x),\a(y))\\
      &=\ell(\a(x))f(\a(y))+r(\a(y))f(\a(x))-f(\a(x)\a(y))\\
      &=\ell(\a(x))\f(f(y))+r(\a(y))\f(f(x))-\f(f(xy))\\
      &=\f(\ell(x)f(y))+\f(r(y)f(x))-\f(f(xy))\\
      &=\f \circ \d^1f(x,y).
\end{align*}
and
\begin{align*}
    & \d^2f \circ \a(x,y,z)=\d^2f(\a(x),\a(y),\a(z))\\
=& r(\a\beta(z))f(\a\beta(x),\alpha^2(y))-
   \ell(\alpha^2\beta(x))f(\alpha^2(y),\a(z)) \\
   +&r(\a\beta(z))f(\a\beta(y),\alpha^2(x))-
   \ell(\alpha^2\beta(y))f(\alpha^2(x),\a(z))  \\
   +&f(\a\beta(x)\alpha^2(y),\a\beta(z))-
   f(\alpha^2\beta(x),\alpha^2(y)z) \\
   +&  f(\a\beta(y)\alpha^2(x),\a\beta(z))-
   f(\alpha^2\beta(y),\alpha^2(x)\a(z))\\
=& r(\a\beta(z))\f f(\beta(x),\alpha(y))-
   \ell(\alpha^2\beta(x))\f f(\alpha(y),z) \\
   +&r(\a\beta(z))\f f(\beta(y),\alpha(x))-
   \ell(\alpha^2\beta(y))\f f(\alpha(x),z)  \\
   +&\f f(\beta(x)\alpha(y),\beta(z))-
\f   f(\alpha\beta(x),\alpha(y)z) \\
   +& \f f(\beta(y)\alpha(x),\beta(z))-
 \f  f(\alpha\beta(y),\alpha(x)z)\\
  =& \f \circ \d^2f(x,y,z).
\end{align*}
The other identities can be showed using the same computation.
\end{proof}
\end{lem}

\begin{prop}
The composite $\delta^2\circ \delta^1$ is zero.
\end{prop}
\begin{proof}

 Let $x, \ y,\  z \in A$ and $f \in C^1_{\alpha, \beta}(A,V)$. Then
\begin{align}
 & \delta^2\circ \delta^1f(x,y,z)= \nonumber \\
  &r(\b(z))\delta^1f(\b(x),\a(y))-
 \ell( \alpha\beta(x))\delta^1f(\alpha(y),z)
   +r(\beta(z))\delta^1f(\beta(y),\alpha(x))-\ell(\alpha\beta(y))\delta^1f(\alpha(x),z)  \nonumber\\
   +& \delta^1f(\beta(x)\alpha(y),\beta(z))-
   \delta^1f(\alpha\beta(x),\alpha(y)z)
+  \delta^1f(\beta(y)\alpha(x),\beta(z))-
   \delta^1f(\alpha\beta(y),\alpha(x),z)  \nonumber\\
=&r(\b(z))\ell(\b(x))f(\a(y))+r(\b(z))r(\a(y))f(\b(x))-r(\b(z))f(\b(x)\a(y))   \nonumber\\
-&   \ell(\a\b(x))\ell(\a(y))f(z)-\ell(\a\b(x))r(z)f(\a(y))+\ell(\a\b(x))f(\a(y)z)  \nonumber\\
+&r(\b(z))\ell(\b(y))f(\a(x))+r(\b(z))r(\a(x))f(\b(y))-r(\b(z))f(\b(y)\a(x))  \nonumber\\
-&   \ell(\a\b(y))\ell(\a(x))f(z)-\ell(\a\b(y))r(z)f(\a(x))+\ell(\a\b(y))f(\a(x)z)  \nonumber\\
+&\ell(\b(x)\a(y))f(\b(z))+r(\b(z))f(\b(x)\a(y))-f((\b(x)\a(y))\b(z)) \nonumber \\
-& \ell(\a\b(x))f(\a(y)z)-r(\a(y)z)f(\a\b(x))+f(\a\b(x)(\a(y)z))  \nonumber\\
+&\ell(\b(y)\a(x))f(\b(z))+r(\b(z))f(\b(y)\a(x))-f((\b(y)\a(x))\b(z))  \nonumber\\
-& \ell(\a\b(y))f(\a(x)z)-r(\a(x)z)f(\a\b(y))+f(\a\b(y)(\a(x)z))  \nonumber\\
=&r(\b(z))\ell(\b(x))f(\a(y))+r(\b(z))r(\a(x))f(\b(y)) \label{1} \\
-&\ell(\a\b(x))r(z)f(\a(y))  -r(\a(x)z)f(\a\b(y))  \label{2} \\
  +&r(\b(z))\ell(\b(y))f(\a(x))  +  r(\b(z))r(\a(y))f(\b(x))  \label{3} \\
-&\ell(\a\b(y))r(z)f(\a(x))  -r(\a(y)z)f(\a\b(x))   \label{4}  \\
+&\ell(\b(x)\a(y))f(\b(z)) -\ell(\a\b(x))\ell(\a(y))f(z)   \label{5} \\
+&\ell(\b(y)\a(x))f(\b(z))  -  \ell(\a\b(y))\ell(\a(x))f(z)    \label{6}  \\
-&f(as_{\a,\b}(\b(x),\a(y),z)+as_{\a,\b}(\b(y),\a(x),z)   \label{7}
\end{align}
We have \eqref{1}+\eqref{2}=$0$ and \eqref{3}+\eqref{4}=$0$, using \eqref{rep3}. Moreover,
\eqref{5}+ \eqref{6}=$0$, using \eqref{rep1}.  In addition, \eqref{7}=$0$, since $A$ is left BiHom-alternative.

Then, we get $\d^2 \circ \d^1=0$
\end{proof}

\begin{df}
The coboundary operator $\delta^3: C^3_{\a,\b}(A;V)  \rightarrow
C^4_{\a,\b,}(A;V)$ is defined by
\begin{align*}
 & \delta^3f(x_1,x_2,x_3,x_4) =
\ell(\a(x_1)) f(\b(x_2),\b( x_3),\b(x_4))-
\ell(\a(x_1)) f(\b(x_3), \b(x_2),\b(x_4)) \\
+ & r(\b( x_4)) f(\a(x_1), \a(x_2), \a(x_3)) -
r(\b(x_4))f(\a(x_2), \a(x_1),\a(x_3))   \\
- &f(\a(x_1)\b(x_2), x_3, x_4 )-
f( \a(x_2)  \b(x_3), x_1, x_4 )
 +f (x_1, \a(x_2) \b(x_3), x_4 )\\
 +&
 f(x_3, a(x_1)\b(x_2),  x_4 )
-f( x_1, x_2, \a(x_3)\b(x_4) )+
f( x_2, x_1, \a(x_3)\b(x_4) ).
\end{align*}
\end{df}
\begin{lem}  The operator $\delta^3$ is well defined, that is
$$\d^3f \circ \a^{\otimes4} =\f \circ \d^3f\  \textrm{and}\  \d^3f \circ \b^{\otimes4} =\p \circ \d^3f.$$
\end{lem}
\begin{proof}
We can check these identities by a direct computation.
\end{proof}
\begin{prop}
The composite $\delta ^{3}\circ \delta ^{2}$ is zero.
\end{prop}
\begin{proof}
Let $x_1,x_2,x_3,,x_4 \in A$ and  $f \in \mathcal{C}^{2}({A ;V})$,
Then, by substituting $f$ with $\delta ^{2}f$ in the previous formula and rearranging the terms we get
\begin{align*}
& \delta^3(\d^2f)(x_1,x_2,x_3,x_4) =
\ell(\a(x_1) ) \d^2f(\b(x_2),\b( x_3),\b(x_4))-
\ell(\a(x_1)) \d^2f(\b(x_3), \b(x_2),\b(x_4)) \\
&
+r(\b( x_4) )\d^2f(\a(x_1), \a(x_2), \a(x_3)) -
r(\b(x_4)) \d^2f(\a(x_2), \a(x_1),\a(x_3))   \\
&
-\d^2f( \a(x_1)\b(x_2), x_3, x_4 )-
\d^2f( \a(x_2)\b(x_3), x_1, x_4 )\\
&
 +\d^2f (x_1, \a(x_2)\b(x_3), x_4 )+
\d^2 f(x_3, \a(x_1)\b(x_2),  x_4 )\\
&
-\d^2f( x_1, x_2, \a(x_3)\b(x_4) )+
\d^2f( x_2, x_1, \a(x_3) \b(x_4) ).\\
&= \ell(\a(x_1) )(\d^2f(\b(x_2),\b( x_3),\b(x_4))-\d^2f(\b(x_3), \b(x_2),\b(x_4)))\\
&+r(\b(x_4)) ( \d^2f(\a(x_1), \a(x_2), \a(x_3)-\d^2f(\a(x_2), \a(x_1),\a(x_3))  \\
& +\d^2 f(x_3, \a(x_1)\b(x_2),  x_4 )--\d^2f( \a(x_1)\b(x_2), x_3, x_4 )\\
&+\d^2f (x_1, \a(x_2)\b(x_3), x_4 )-\d^2f( \a(x_2) \b(x_3), x_1, x_4 )\\
&+\d^2f( x_2, x_1,\a(x_3) \b(x_4) ).  -\d^2f( x_1, x_2,\a(x_3) \b(x_4))\\
&=0 \quad\ (\text{since}\  \delta^2f=\delta^2f \circ \tau_{12}).
\end{align*}

\end{proof}
One can complete the complex by considering $\d^p=0$, for $p>3$.
For $n=1,2,3$, the map $f \in C_{\alpha,\beta}^n(A;V)$ is called a $n$-BiHom-cocycle (or simply a $n$-cocycle), if $\delta^nf=0$.
We denote by $Z_{\alpha,\beta}^n(A;V)$ the subspace spanned by $n$-BiHom-cocycles  and by $B_{\alpha,\beta}^n(A;V)=\delta^{n-1}(C_{\alpha,\beta}^{n-1}(A;V))$.
Since $\delta^2\circ \delta^1=0$ and  $\delta^3\circ \delta^2=0$, then $B^2_{\alpha,\beta}(A;V)$ is a subspace of $Z_{\alpha,\beta}^2(A;V)$ and
$B^3_{\alpha,\beta}(A;V)$ is a subspace of $Z_{\alpha,\beta}^3(A;V)$.
Hence, we can define  cohomology spaces of $(A,\mu,\alpha,\b)$ as
\begin{align*}
&H_{\a,\b}^2(A;V)=
\frac{Z_{\a,\b}^2(A;V)}{B^2_{\a,\b}(A;V)},
\qquad H_{\a,\b}^3(A;V)= \frac{Z_{\a,\b}^3(A;V)}{B^3_{\a,\b}(A;V)}.
\end{align*}

 \section{Deformations of left BiHom-alternative algebras}
We develop, in this section, a deformation theory for left BiHom-alternative algebras by analogy with Gerstenhaber
deformations \cite{Gerstenhaber1, Gerstenhaber2, Gerstenhaber3, Gerstenhaber4}. Heuristically, a formal deformation of an algebra $A$ is a
one-parameter family of multiplication (of the same sort) obtained by perturbing the multiplication of $A$.
\subsection{Definition}
Suppose that $(A,\mu, \alpha,\beta)$ is a  left BiHom-alternative algebra. Let $\mathbb{K}[[t]]$ be the ring of formal power
series over $\mathbb{K}$. Suppose that $A[[t]]$ is the set of formal power series over $A$. Then for a $\mathbb{K}$-bilinear map $f:A\times A \rightarrow A$, it is natural to extend it to be a $\mathbb{K}[[t]]$-bilinear map
$f:A[[t]]\times A[[t]]\rightarrow A[[t]]$ by
$$
f\bigg(\displaystyle\sum_{i\geq0}x_it^i,\displaystyle\sum_{j\geq0}y_jt^j\bigg)=\displaystyle\sum_{i,j\geq0}f(x_i,y_j)t^{i+j}.
$$
\begin{df}

Suppose that $(A,\mu, \alpha,\beta)$ is a left   BiHom-alternative algebra. A one-parameter
formal deformation of $(A,\mu, \alpha,\beta)$ is a formal power series $d_t:A[[t]]\times A[[t]]\rightarrow A[[t]]$ of the form
$$
d_t(x,y)=\displaystyle\sum_{i\geq0}d_i(x,y)t^i=d_0(x,y)+d_1(x,y)t + d_2(x,y)t^2 +\dots,$$
where each $d_i$ is a $\mathbb{K}$-bilinear map $d_i:A\times A\rightarrow A$ (extended to be $\mathbb{K}[[t]]$-bilinear)
and $d_0(x,y)=\mu(x,y)$, such that the following identities must be
satisfied:
\begin{eqnarray}
&&d_t(\alpha(x),\alpha(y))= \alpha\circ d_t(x, y),\label{deformation1}\\
&&d_t(\beta(x),\beta(y))= \beta\circ d_t(x,y),\label{deformation2}\\
&&d_t(d_t(\b(x),\a(y)),\b(z))-d_t(\a\b(x),d_t(\a(y),z))\nonumber\\
&&+d_t(d_t(\b(y),\a(x)),\b(z))-d_t(\a\b(y),d_t(\a(x),z))=0\label{deformation3}.
\end{eqnarray}
Conditions \eqref{deformation1}-\eqref{deformation3} are called  deformation equations of a  left BiHom-alternative algebra.

\end{df}

Note that $A[[t]]$ is a module over $\mathbb{K}[[t]]$ and $d_t$ defines the bilinear multiplication on
$A[[t]]$ such that $A_t=(A[[t]], d_t, \alpha,\beta)$ is a  left BiHom-alternative algebra.

\subsection{Deformation equations and Obstructions}
Now we investigate the deformation equations \eqref{deformation1}-\eqref{deformation3}.
Conditions \eqref{deformation1} and \eqref{deformation2} are equivalent to the following equations:
\begin{gather}
d_i(\alpha(x),\alpha(y))=\alpha\circ d_i(x,y),\label{deformation1.1}\\
d_i(\beta(x),\beta(y))=\beta\circ d_i(x,y).\label{deformation1.2},
\end{gather}
respectively, for $i=0,1,2,\cdots.$  Expanding the
left side of the equation \eqref{deformation3} and collecting the coefficients of $t^k$, it
yields
\begin{align*}
\sum_{i+j=k,i,j\geq0}&\big(d_i(d_j(\b(x),\a(y)),\b(z))-d_i(\a\b(x),d_j(\a(y),z))\\
&+d_i(d_j(\b(y),\a(x)),\b(z))-d_i(\a\b(y),d_j(\a(x),z))\big)=0,\; k=0,1,2\cdots.
\end{align*}
This equation  gives the necessary and sufficient conditions for the
 left BiHom-alternativity of $d_t$. It may be written
\begin{align}\label{deformation 4}
\sum_{i=0}^k &\big(d_i(d_{k-i}(\b(x),\a(y)),\b(z))-d_i(\a\b(x),d_{k-i}(\a(y),z)) \nonumber\\
&+d_i(d_{k-i}(\b(y),\a(x)),\b(z))-d_i(\a\b(y),d_{k-i}(\a(x),z))\big)=0.
\end{align}
The first equation $(k = 0)$ is the left BiHom-alternative condition for $\mu_0$.
For $k=1$, we obtain the following equation:
\begin{align*}
  &   \mu(d_1(\beta(x),\alpha(y)),\beta(z))-
   \mu(\alpha\beta(x),d_1(\alpha(y),z))
   +\mu(d_1(\beta(y),\alpha(x)),\beta(z))-\\
  & \mu(\alpha\beta(y),d_1(\alpha(x),z))
   +d_1(\mu(\beta(x),\alpha(y)),\beta(z))-
 \  d_1(\alpha\beta(x),\mu(\alpha(y),z)) \\
  &\qquad +  d_1(\mu(\beta(y),\alpha(x)),\beta(z))-
   d_1(\alpha\beta(y),\mu(\alpha(x),z))=0.
\end{align*}
This means that  $d_1$ must be a $2$-cocycle for the left BiHom-alternative algebra cohomology defined above that is  $d_1\in Z_{\a,\b}^2(A,A)$.
More generally, suppose that $d_p$ be the first non-zero coefficient after $\mu_0$ in the deformation $d_t$. This $d_p$ is
called the infinitesimal of $d_t$
\begin{thm}
The map $d _p$  is a $2$-cocycle of a left BiHom-alternative algebra
cohomology of $A$  with coefficient in itself.
\end{thm}
\begin{proof}
In the equation \eqref{deformation 4} make the following substitution $k=p$ and $d_1=\cdots =d _{p-1}=0$.
\end{proof}
\begin{df}
Let $(A,\mu,\a,\b)$ be a left BiHom-alternative algebra and $d_1$ be an element in $Z^2_{\a,\b}(A,A)$.  The $2$-cocycle  $d_1$ is said integrable if there exists a family $(d_i)_{i \geq 0}$ such that $d_t=\sum_{i\geq 0} d_i t^i$ defines a formal deformation $A_t=(A[[t]],d_t,\a,\b)$ of $A$.
\end{df}
The integrability of $d _p$ implies an infinite sequence of relations
which may be interpreted as the vanishing of the obstruction to the
integration of $d_p$.

For an arbitrary $k$, with $k>1,$ the $k^{th}$ equation of the system \eqref{deformation 4} may be written%
\begin{eqnarray*}
\delta ^2d_k\left( x,y,z\right) &=& \sum_{i=1}^{k-1}d _i\left( d
_{k-i}\left( \b(x), \a(y)\right)  \b(z)\right) -d _i\left( \a\b(x) , d _{k-i}\left(
\a(y), z\right) \right)\\
&+&d _i\left( d
_{k-i}\left( \b(y),\a(x)\right), \b(z)\right) -d _i\left( \a\b(y), d_{k-i}\left(
\a(x), z\right) \right).
\end{eqnarray*}
Suppose that the truncated deformation  of order $m-1$
$$d _t=d _0+td_1+t^2d_2+\cdots +t^{m-1}d _{m-1}$$
satisfies the deformation equation. The truncated
deformation is extended to a deformation of order $m$,  $d_t=d
_0+td _1+t^2d_2+ \cdots +t^{m-1}d _{m-1}+t^md _m$, satisfying   the
deformation equation if
\begin{eqnarray*}
\delta ^2d_m\left( x,y,z\right) &=&   \sum_{i=1}^{m-1}d _i\left( d
_{m-i}\left( \b(x), \a(y)\right)  \b(z)\right) -d _i\left( \a\b(x) , d _{m-i}\left(
\a(y), z\right) \right)\\
&+&d _i\left( d
_{m-i}\left( \b(y),\a(x)\right), \b(z)\right) -d _i\left( \a\b(y), d_{m-i}\left(
\a(x), z\right) \right).
\end{eqnarray*}

The right-hand side of this equation is called the {\it obstruction }to
finding $d _m$ extending the deformation.

We define  a  diamond operation on $2$-cochains by
\begin{align*}
 &  d _i\diamond d _j\left( x, y, z\right) =d _i\left( d
_j\left( \b(x), \a(y)\right)  \b(z)\right) -d _i\left( \a\b(x) , d _j\left(
\a(y), z\right) \right)\\
& +d _i\left( d_j\left( \b(y),\a(x)\right), \b(z)\right) -d _i\left( \a\b(y), d_j \left(\a(x), z\right) \right).
\end{align*}
Then  the obstruction may be written
$$
\sum_{i=1}^{m-1}{d_i\diamond d _{m-i}}\text{ or }
\sum_{ i+j=m \ \ i,j\neq m}{ d _i\diamond d _j }.
$$

\noindent
A straightforward computation gives the following result.
\begin{thm}
The obstructions are left BiHom-alternative 3-cocycles.
\end{thm}
We cite the following observations
\begin{enumerate}
\item The cohomology class of the element $\displaystyle\sum_ {i+j=m,\ \ i,j\neq m}
d_i\diamond  d _j$ is the first obstruction to the integrability of $d _m$.

Let us consider now how to extend an infinitesimal deformation to a deformation of order $2$. Suppose $m=2$ and $d_t=d _0+td _1+t^2d _2$. The deformation
equation of the truncated deformation of order $2$ is equivalent to the finite system
$$
\left\{
\begin{array}{lll}
d_0\diamond  d _0&=&0\quad \left( d _0
\text{ is left BiHom-alternative}\right)  \\ \delta^2 d _1&=&0\quad \left( d _1\in
Z^2\left(A,A\right) \right)  \\
d _1\diamond  d_1&=&\delta^2 d _2
\end{array}
\right.
$$

Then $d _1\diamond d _1$ is the first obstruction to integrate $d _1$ and
$d _1\diamond d _1\in Z^3\left( A,A\right) .$

The elements $d_1\diamond d _1$ which are coboundaries permit to extend the deformation of order one
to a deformation of order $2$. But the elements of $H^3\left( A,A\right) $
gives the obstruction to the integration of $d _1$.

\item  If $d _m$ is integrable, then the cohomological class of $\displaystyle \sum_
{i+j=m, \quad i,j\neq m} d _i\ \diamond d_j$ must be $0$.
In the previous example,  $d _1$ is integrable implies $d_1\diamond  d _1=\delta^2 d _2$
which means that the cohomology class of $d _1\diamond d _1$ vanishes.
\end{enumerate}

\begin{cor}
If $H^3\left( A,A\right) =0$, then all obstructions vanish and every $%
\mu _m\in Z^2\left( A,A\right) $ is integrable.
\end{cor}

 \subsection{Equivalent and trivial deformations}
 In this paragraph, we characterize the equivalent and trivial deformations of left BiHom-alternative algebras.
\begin{df}
Let $(A, \mu, \alpha,\beta)$ be a left BiHom-alternative algebra. Suppose that $d_t(x,y,)=\sum_{i\geq0}d_i(x,y)t^i$ and $d_t'(x,y)=\sum_{i\geq0}d_i'(x,y)t^i$ are two one-parameter formal deformations of $(A, \mu, \alpha,\beta)$. They are called equivalent, denoted by $d_t\sim d_t'$, if there is a formal isomorphism of $\mathbb{K}[[t]]$-modules
$$\phi_t(x)=\sum_{i\geq0}\phi_i(x)t^i:(A[[t]],d_t,\alpha,\beta)\longrightarrow (A[[t]],d_t',\alpha,\beta),$$
where each $\phi_i:A\rightarrow A$ is a $\mathbb{K}$-linear map (extended to be $\mathbb{K}[[t]]$-linear) and $\phi_0=Id_{A}$, satisfying
\begin{align*}
&\phi_t\circ\alpha=\alpha\circ\phi_t, \ \phi_t\circ\beta=\beta \circ\phi_t \\
&\phi_t\circ d_t(x,y)=d_t'(\phi_t(x),\phi_t(y)).
\end{align*}
When $d_1=d_2=\cdots=0$, $d_t=d_0$ is said to be the null deformation. A 1-parameter formal deformation $d_t$ is called trivial if $d_t\sim d_0$. A left BiHom-alternative algebra $(A, \mu, \alpha,\beta)$ is called analytically rigid, if every $1$-parameter formal deformation $d_t$ is trivial.
\end{df}
\begin{thm}
 Let  $d_t(x, y)=\sum_{i>0}d_i(x,y)t^i$
 and $d'_t(x, y)=\sum_{i\geq0}d'_i(x,y)t^i$ be equivalent  $1$-parameter formal deformations of a left BiHom-alternative algebra $(A,\mu, \a,\b)$. Then $d_1$ and $d'_1$ belong to the
same cohomology class in $H_{\a,\b}^2(A,A)$.
\end{thm}
\begin{proof}
Suppose that $\phi_t(x)=\sum_{i\geq 0}\phi_i(x)t^i$ is the formal $\mathbb{K}[[t]]$-module  isomorphism
such that $\phi_t \circ \a=\a\circ \phi_t, \phi_t \circ \b=\b\circ \phi_t$ and
\begin{eqnarray*}
\sum_{i\geq 0}\phi_i\Big(\sum_{j>0}d_j(x,y)t^j\Big)t^i=\sum_{i\geq 0}d'_{i}\Big(\sum_{k\geq 0}\phi_k(x)t^k, \sum_{k\geq 0}\phi_l(y)t^l\Big)t^i.
\end{eqnarray*}
It follows that
\begin{eqnarray*}
\sum_{i+j=n}\phi_i\big(d_j(x,y)\big)t^{i+j}=\sum_{i+k+l=n}d'_{i}\big(\phi_k(x),\phi_l(y)\big)t^{i+k+l}.
\end{eqnarray*}

In particular,
\begin{eqnarray*}
\sum_{i+j=1}\phi_i\big(d_j(x,y)\big))t^{i+j}=\sum_{i+k+l=1}d'_{i}\big(\phi_k(x),\phi_l(y)\big)t^{i+k+l}.
\end{eqnarray*}
Since $\f_0=id$, we obtain
\begin{eqnarray*}
d_1(x,y)+\phi_1(\mu(x,y))=\mu(\phi_1(x),y)+\mu(x, \phi_1(y))+d'_1(x,y).
\end{eqnarray*}
Then $d_1-d'_1\in B_{\a,\b}^2(A,A)$. Therefore, $d_1$ and $d'_1$ are cohomologous.
\end{proof}

\begin{thm}
Let $(A,\mu, \a,\b)$ be a  left BiHom-alternative algebra such that  $H_{\a,\b}^2(A,A)=0$. Then $(A,\mu, \a,\b)$ is analytically rigid.
\end{thm}

\begin{proof}
 Let $d_t$ be a one-parameter formal deformation of $(A,\mu, \a,\b)$. Suppose that $d_t=d_0+\sum_{i\geq n}d_it^i$. Then
\begin{eqnarray*}
\delta^2d_n=d_1\diamond d_{n-1}+d_2 \diamond d_{n-2}+\cdot\cdot\cdot+d_{n-1} \diamond d_0=0,
\end{eqnarray*}
that is $d_n\in Z_{\a,\b}^2(A,A)=B_{\a,\b}^2(A,A)$. It follows that there exists $f_n\in C_{\a,\b}^1(A,A)$ such that $d_n=\delta^1f_n$.

Let $\phi_t=id_{A}-t^nf_n: (A[[t]], d'_t, \a,\b)\rightarrow (A[[t]], d_t, \a,\b) $, where
$d'_t(x,y)=\phi_t^{-1}d_t(\phi_t(x), \phi_t(y))$.  Here, $\phi_t$ is a linear isomorphism
since
\begin{eqnarray*}
\phi_t\circ \sum_{i\geq 0}f_n^it^{in}=\sum_{i\geq 0}f_n^it^{in}\circ \phi_t=id_{A[[t]]}.
\end{eqnarray*}
 Moreover, we have $\phi_t\circ \a=\a\circ\phi_t $ and $\phi_t\circ \b=\b\circ\phi_t $.

It is straightforward to prove that $d'_t$ is a one-parameter formal deformation of $(A,\mu, \a,\b)$ and $d_t\sim d'_t$. Assume that $d'_t(x,y)=\sum_{i\geq 0}d'_i(x,y)t^i$ and use the fact that $\f_t \circ d'_t(x,y)=d_t(\f_t(x),\f_t(y))$. Then
\begin{eqnarray*}
(id_{A}-f_nt^n)\Big(\sum_{i\geq 0}d'_i(x,y)t^i\Big)
=  \Big(d_0+\sum_{i\geq n}d_it^i\Big)(x-f_n(x)t^n,y-f_n(y)t^n),
\end{eqnarray*}
i.e.,
\begin{eqnarray*}
&&d'_0(x, y)+\sum_{i\geq 1}d'_i(x,y)t^i-\sum_{i \geq 0} f_n \circ d'_i(x,y)t^{i+n}\\
&=&d_0(x,y)-(d_0(f_n(x),y)+d_0(x, f_n(y)))t^n+d_0(f_n(x), f_n(y))t^{2n}\\
&&+\sum_{i\geq n}d_{i}(x,y)t^i-\sum_{i\geq n}(d_{i}(f_n(x),y)+d_{i}(x,f_n(y) )t^{i+n}+\sum_{i\geq n}d_i(f_n(x), f_n(y))t^{i+2n}.
\end{eqnarray*}
Then we have $d_1'=d_2'=...=d_{n-1}'=0$. Since $d_0=d'_0$, then
\begin{eqnarray*}
d_n'(x,y)-f_n \circ d_0(x,y) =-(d_0(f_n(x),y)+d_0(x, f_n(y)))+d_n(x,y).
\end{eqnarray*}
Hence $d_n'=d_n-\delta^1f_n=0$ and $d'_t(x,y)=d'_0+\sum_{i\geq n+1}d'_i(x,y)t^i$. By induction, this procedure
ends with $d_t \sim d_0$, so., $(A,\mu, \a,\b)$ is analytically rigid.
\end{proof}

\section{Extensions of BiHom-alternative algebras}
\subsection{Dual representation}
In this section, we construct the dual representation of a representation of a BiHom-alternative algebra
without any additional condition. This is nontrivial and to our knowledge, people needed to add a very
strong condition to obtain a representation on the dual space in  former studies which  restricts
its development.

Let $(V,\ell,r, \f,\p)$ be a regular representation of a BiHom-alternative algebra $(A,\mu,\a,\b)$. Define $\ell^*,r^*:A\longrightarrow \textrm{gl}(V^*)$ as usual by
$$\langle \ell^*(x)(\xi),u\rangle=\langle\xi,\ell(x)(u)\rangle \quad \text{and}\ \langle r^*(x)(\xi),u\rangle=\langle\xi,r(x)(u)\rangle$$ for all $x\in A,u\in V,\xi\in V^*.$
However, in general $(r^*,\ell^*)$ is not a representation of $A$. Define $r^\star,\ell^\star:A\longrightarrow \textrm{gl}(V^*)$ by
\begin{equation}\label{eq:new1}
 \ell^\star(x)(\xi):=\ell^*(\a^{-1}\b^2(x))\big{(}(\f^{-1}\p^{-1})^*(\xi)\big{)}
\end{equation} and \begin{equation}\label{eq:new2}
 r^\star(x)(\xi):=r^*(\a^2\b^{-1}(x))\big{(}(\f^{-1}\p^{-1})^*(\xi)\big{)}
\end{equation} for all $x\in A,\xi\in V^*$.
More precisely, we have
\begin{eqnarray}\label{eq:new1gen}
&&\langle\ell^\star(x)(\xi),u\rangle=\langle\xi,\ell(\a^{-2}\b(x))(\f^{-1}\p^{-1}(u))\rangle\\ \label{eq:new2gen}&&
\langle r^\star(x)(\xi),u\rangle=\langle\xi,r(\a\b^{-2}(x))(\f^{-1}\p^{-1}(u))\rangle
\end{eqnarray}for all $x\in A, u\in V, \xi\in V^*$.
\begin{thm}\label{lem:dualrep}
 Let $(V,\ell,r,\f,\p)$ be a representation of a BiHom-alternative algebra $(A,\mu,\a,\b)$. Then $(V^*,r^\star,\ell^\star,(\f^{-1})^*,(\p^{-1})^*)$ is a representation of $(A,\mu,\a,\p)$ where the linear maps $\ell^\star,r^\star:A\longrightarrow gl(V^*)$ are defined above by \eqref{eq:new1} and  \eqref{eq:new2} .
\end{thm}
\begin{proof}
For all $x\in A,\xi\in V^*$,  we have
{\small\begin{eqnarray*}
\ell^\star(\a(x))((\f^{-1})^*(\xi))=\ell^*(\b^2(x))(\f^{-2}\p^{-1})^*(\xi)=(\f^{-1})^*(\ell^*(\a^{-1}\b^2(x))(\f^{-1}\p^{-1})^*(\xi))=(\f^{-1})^*(\ell^\star(x)(\xi))
\end{eqnarray*}}
and 
{\small\begin{eqnarray*}
\ell^\star(\b(x))((\p^{-1})^*(\xi))=\ell^*(\a^{-1}\b^3(x))(\f^{-1}\p^{-2})^*(\xi)=(\p^{-1})^*(\ell^*(\a^{-1}\b^2(x))(\f^{-1}\p^{-1})^*(\xi))=(\p^{-1})^*(\ell^\star(x)(\xi)).
\end{eqnarray*}}
Similarly, we can check that
\begin{eqnarray*}
r^\star(\a(x))((\f^{-1})^*(\xi))=(\f^{-1})^*(r^\star(x)(\xi)),
\end{eqnarray*}
\begin{eqnarray*}
r^\star(\b(x))((\p^{-1})^*(\xi))=(\p^{-1})^*(r^\star(x)(\xi)).
\end{eqnarray*}
On the other hand, by the Eqs \eqref{rep1}, \eqref{rep2} and \eqref{eq:new1gen} , for all $x\in A,\xi\in V^*$ and $u\in V$, we have
\begin{eqnarray*}
&&\langle  \ell^\star(\b(x)\a(x))(\f^{-1})^*(\xi)- \ell^\star(\a\b(x))\ell^\star(\b(x))(\xi),u\rangle\\
&=&\langle  \ell^\star(\b(x)\a(x))(\f^{-1})^*(\xi),u\rangle-\langle   \ell^\star(\a\b(x))\ell^\star(\b(x))(\xi),u\rangle\\
&=&\langle(\f^{-1})^*(\xi),\ell(\a^{-2}\b^2(x)\a^{-1}\b(x))(\f^{-1}\p^{-1}(u))\rangle-\langle   \xi,\ell(\a^{-2}\b^2(x))\ell(\a^{-2}\b(x))(\f^{-2}\p^{-2}(u))\rangle\\
&=&\langle\xi,\ell(\a^{-3}\b^2(x)\a^{-2}\b(x))(\f^{-2}\p^{-1}(u))-\ell(\a^{-2}\b^2(x))\ell(\a^{-2}\b(x))(\f^{-2}\p^{-2}(u))\rangle=0,
\end{eqnarray*}
Using Eqs \eqref{rep2}, \eqref{rep2} and \eqref{eq:new2gen} , for all $x\in A,\xi\in V^*$ and $u\in V$, we have
\begin{eqnarray*}
&&\langle  r^\star(\b(x)\a(x))(\p^{-1})^*(\xi)- r^\star(\a\b(x))r^\star(\a(x))(\xi),u\rangle\\
&=&\langle  r^\star(\b(x)\a(x))(\p^{-1})^*(\xi),u\rangle-\langle   r^\star(\a\b(x))r^\star(\a(x))(\xi),u\rangle\\
&=&\langle(\p^{-1})^*(\xi),r(\a\b^{-1}(x)\a^2\b^{-2}(x))(\f^{-1}\p^{-1}(u))\rangle-\langle   \xi,r(\a^2\b^{-2}(x))r(\a\b^{-2}(x))(\f^{-2}\p^{-2}(u))\rangle\\
&=&\langle\xi,r(\a\b^{-2}(x)\a^2\b^{-3}(x))(\f^{-1}\p^{-2}(u))-r(\a^2\b^{-2}(x))r(\a\b^{-2}(x))(\f^{-2}\p^{-2}(u))\rangle=0.
\end{eqnarray*}
Similarly, we can obtain the other identities.
Therefore, $(V^*,r^\star,\ell^\star,(\f^{-1})^*,(\p^{-1})^*)$ is a representation of  $(A,\mu,\a,\b)$.
\end{proof}

\begin{lem}\label{lem:dualdual}
Let $(V,\ell,r,\f,\p)$ be a representation of BiHom-alternative  algebra  $(A,\mu,\a,\b)$. Then we have $$\ (\ell^\star)^\star=\ell\circ \a^{-3}\b^3\ \text{and} \ \ (r^\star)^\star=r\circ \a^3\b^{-3} .$$
\end{lem}
\begin{proof}
Let $x\in A,u\in V,\xi\in V^*$, we have
\begin{eqnarray*}
\langle\xi, (\ell^\star)^\star(x)(u)\rangle
&=&\langle\xi,(\ell^\star)^*(\a^{-1}\b^2(x))(\f\p(u))\rangle
=\langle\ell^\star(\a^{-1}\b^2(x))(\xi),\f\p(u)\rangle\\
&=&\langle\xi,\ell(\a^{-3}\b^{3}(x))(u)\rangle,
\end{eqnarray*}
which implies that $(\ell ^\star)^\star=\ell \circ \a^{-3}\b^3$.
Similarly, we can check 
$(r^\star)^\star=r\circ\a^3\b^{-3}  .$
\end{proof}

\begin{cor}\label{lem:rep}
 Let $(A,\mu,\a,\b)$ be a BiHom-alternative algebra. Then $R^\star,L^\star:A\longrightarrow gl(A^*)$  defined by
 \begin{equation}
  L^\star(x)(\xi)=L^*(\a^{-1}\b^2(x))(\a^{-1}\b^{-1})^*(\xi)\ \text{and} \ R^\star(x)(\xi)=R^*(\a^2\b^{-1}(x))(\a^{-1}\b^{-1})^*(\xi)
 \end{equation} for all $x\in A, \xi\in A^*$,
 is a representation of the BiHom-alternative algebra $(A,\mu,\a,\b)$ on $A^*$ with respect to $\big((\a^{-1})^*,(\b^{-1})^*\big)$.  It is called the coadjoint representation.
\end{cor}

Using the coadjoint representation $(R^\star,L^\star)$, we can obtain a semidirect product BiHom-alternative algebra structure on $A\oplus A^*$.
\begin{cor}
Let $(A,\mu,\a,\b)$ be a BiHom-alternative algebra. Then there is a natural BiHom-alternative  algebra $(A\oplus A^*,\mu_{A\oplus A^*},\a\oplus(\a^{-1})^*,\b\oplus(\b^{-1})^*)$, where the BiHom-alternative product $\mu_{A\oplus A^*}$ is given by
\begin{equation}\label{eq:semidirectproduct}
  \mu_{A\oplus A^*}(x+\xi,y+\eta)=\mu(x,y)+R^\star(\eta)+L^\star(x)(y)(\xi)
\end{equation}
for all $x,y\in A, \xi,\eta\in A^*$.
\end{cor}

\subsection{Central Extensions and $T_\theta$-Extensions of BiHom-alternative algebras}
In this section, we deal with extensions of BiHom-alternative algebras. \\
An algebra $\tilde{A}$ is  an extension of a BiHom-alternative algebra $A$  by $V$ if there is an exact sequence $$0 \longrightarrow V  \xrightarrow{\ \ i\ \ } \tilde{A}  \xrightarrow{\ \ \pi\ \ } A \longrightarrow 0$$


A central extension of a BiHom-alternative algebra $(A,\mu,\a,\b)$ is an extension in which the annihilator of $\tilde{A}$ contains $V$. 

\begin{exa}
Let $(A,\cdot,\a,\b)$ be a BiHom-alternative algebra, $V$ a vector space and $\omega:A\times A\to V$ be a bilinear map. We define the following multiplication on the vector space
direct sum $A\oplus V$ by
\begin{align}
& (x+u) \bullet (y+v)= xy + \omega(x,y),\ \forall \ x,y\in A,\ u,v\in V.
\end{align}
Define the linear maps $\overline{\a}, \overline{\b}: A \oplus V \to A \oplus V$ by 
$$\overline{\a}(x+u)=\a(x)+u,\quad
\overline{\b}(x+u)=\b(x)+u,\ \forall x \in A, \ u \in V.$$
It is clear that $\overline{\a}\overline{\b}= \overline{\b}\overline{\a}$. 
Then $(A \oplus V,\overline{\a}, \overline{\b})$ is a BiHom-alternative algebra if and only if 
\begin{align*}
&\overline{\a}\big((x+u) \bullet (y+v)\big)=\big(\overline{\a}(x+u) \bullet \overline{\a}(y+v)\big),\\&\overline{\b}\big((x+u) \bullet (y+v)\big)=\big(\overline{\b}(x+u) \bullet \overline{\b}(y+v)\big),\\
&\big(\overline{\b}(x+u) \bullet \overline{\a}(x+u)\big)\bullet\overline{\b}(y+v)-\overline{\a}\overline{\b}(x+u) \bullet\big( \overline{\a}(x+u)\bullet(y+v)\big)=0,\\&\big((x+u) \bullet \overline{\b}(y+v)\big)\bullet\overline{\a}\overline{\b}(y+v)-\overline{\a}(x+u) \bullet\big( \overline{\b}(y+v)\bullet\overline{\a}(y+v)\big)=0, 
\end{align*}
which are  equivalent  to
\begin{align}
&\omega(\a(x),\a(y))=\omega(\b(x),\b(y))=\omega(x,y),\\
&\omega(\b(x)\a(x),\b(y))=\omega(\a\b(x),\a(x)y),\\
&\omega(x\b(y),\a\b(y))=\omega(\a(x),\b(y)\a(y)).
\end{align}
In this case, we have an exact sequence
$$0 \longrightarrow V  \xrightarrow{\ \ i\ \ } A \oplus V  \xrightarrow{\ \ \pi\ \ } A \longrightarrow 0$$
$ker(\pi)=\{0\} \oplus V \subset Ann(A \oplus V)$. Then the extension is central.  It is called the central extension of $A$ by $V$ via $\omega$. \\
Note that if we consider $V$ as a trivial bimodule of $A$, then $A \oplus V$ is a central extension of $A$ by $V$ via $\omega$ if and only if $\omega \in Z^2(A;V)$. 
\end{exa}

The concept of $T_\theta$-Extensions was developed first in \cite{Bordemann}. We aim in the sequel to discuss it for BiHom-alternative algebras.\\
\begin{prop}
Let $(A,\cdot,\a,\b)$ be a BiHom-alternative algebra and $(V,\ell,r,\f,\p)$ be a bimodule of $A$. Let $\theta: A \times  A \to V$ be a bilinear map.
Define on the direct sum $A \oplus V$ the product
\begin{align}
& (x+u) \circ (y+v)= xy + \ell(x)v+r(y)u + \theta(x,y).
\end{align}
Then $(A \oplus V,\circ,\a+\f,\b+\p)$ is a BiHom-alternative algebra if and only if
the following conditions holds
{\small
\begin{align}
& \f \theta(x,y)= \theta (\a(x),\a(y)),\ \p \theta(x,y)=\theta(\b(x),\b(y)),  \\
& \nonumber \theta(\b(x)\a(y),\b(z))+r(\b(z))\theta(\b(x),\a(y))-\theta(\a\b(x),\a(y)z)-\ell(\a\b(x))\theta(\a(y),z)\\
+&\theta(\b(y)\a(x),\b(z))+r(\b(z))\theta(\b(y),\a(x))-\theta(\a\b(y),\a(x)z)-\ell(\a\b(y))\theta(\a(x),z)  =0, \label{left cocycle}\\
&\nonumber  \theta(x \b(y),\a\b(z))+r(\a\b(z))\theta(x,\b(y))-\theta(\a(x),\b(y)\a(z))-\ell(\a(x))\theta(\b(y),\a(z))  \\
+&  \theta(x \b(z),\a\b(y))+r(\a\b(y))\theta(x,\b(z))-\theta(\a(x),\b(z)\a(y))-\ell(\a(x))\theta(\b(z),\a(y))    =0.
\label{right cocycle}
\end{align}}
\end{prop}
This means that $\theta \in Z^2(A;V)$. 
At theses circumstances, $A \oplus V$ is called the $T_\theta$-extension of $A$ by $V$ via $\theta$. \\
As an application of the previous result, let us consider the dual representation. Let $(V^* ,r^\star,\ell^\star,(\f^{-1})^*,(\p^{-1})^*)$ be the dual representation of
$(V,\ell,r,\f,\p)$. Define a bilinear map $\theta: A \times A \to V^*$ and consider on the vector space $A \oplus V^*$ the product
\begin{align}
& (x+f) \circ (y+g)= xy + r^\star(x)g+\ell^\star(y)f + \theta(x,y).
\end{align}
Then $(A \oplus V^*,\circ,\a+(\f^{-1})^*,\b+(\p^{-1})^*)$ is a BiHom-alternative algebra if and only if
the following conditions holds
{\small
\begin{align}
& (\f^{-1})^* \theta(x,y)= \theta (\a(x),\a(y)),\ (\p^{-1})^* \theta(x,y)=\theta(\b(x),\b(y)),  \\
& \nonumber \theta(\b(x)\a(y),\b(z))+\ell^\star(\b(z))\theta(\b(x),\a(y))-\theta(\a\b(x),\a(y)z)-r^\star(\a\b(x))\theta(\a(y),z)\\
+&\theta(\b(y)\a(x),\b(z))+\ell^\star(\b(z))\theta(\b(y),\a(x))-\theta(\a\b(y),\a(x)z)-r^\star(\a\b(y))\theta(\a(x),z)  =0, \label{left cocycle}\\
&\nonumber  \theta(x \b(y),\a\b(z))+\ell^\star(\a\b(z))\theta(x,\b(y))-\theta(\a(x),\b(y)\a(z))-r^\star(\a(x))\theta(\b(y),\a(z))  \\
+&  \theta(x \b(z),\a\b(y))+\ell^\star(\a\b(y))\theta(x,\b(z))-\theta(\a(x),\b(z)\a(y))-r^\star(\a(x))\theta(\b(z),\a(y))    =0.
\label{right cocycle}
\end{align}}
Under these considerations, $A \oplus V^*$ is called the $T^*_\theta$-extension of $A$ by $V^*$ via  $\theta$. \\
Note that  the $T^*$-extension corresponds to the case of $\theta=0$ in the above construction.

\section{Generalized derivations of BiHom-alternative algebras}
This section is devoted to investigate generalized derivations of BiHom-alternative algebras. Throughout the sequel $A$ denotes a BiHom-alternative algebra $(A,\mu,\a,\b)$.

For any integer $k$ and $l$, denote by $\a^k$ the $k$-times composition of $\a$ and by $\b^l$  the $l$-times composition of $\b$, i.e
$$\a^k=\underbrace{\a \circ \cdots \circ \a}_{k-times}, \quad \b^l=\underbrace{\b \circ \cdots \circ \b}_{l-times}.$$

Let $U$ be the subspace of $End(A)$ defined by
$$U:=\{u \in End(A)\mid u \circ \a= \a \circ u,\  u \circ \b= \b \circ u \}$$
and let $\widetilde{\a},\widetilde{\b}: U \to U$ be two linear maps defined as follow
$$\widetilde{\a}(u)=\a \circ u,\ \widetilde{\b}(u)= \b \circ u,\ \forall u \in U.$$
The space $(U,[\c,\c],\widetilde{\a},\widetilde{\b})$ is a BiHom-Lie algebra where
$[u,v]=u \circ v- v \circ u$ for any $u,v \in U$.


\begin{df}
 A linear map $D: A \to A$ is said to be an $\a^k\b^l$-derivation on $A$, for $k,l \in \mathbb{R}$, if it satisfies the following conditions:
\begin{align}
&[ D , \a]=0, \  [ D , \b]=0,\\
& D(xy)=D(x) \a^k\b^l (y)+ \a^k\b^l(x)D(y), \ \forall x,y \in A.
\end{align}
\end{df}
We denote the set of all $\a^k\b^l$-derivations by  $Der_{\a^k\b^l}(A)$ and
$Der(A)=\bigoplus_{k,l\geq 0} Der_{\a^k\b^l}(A)$.

We can easily show that $Der(A)$ is equipped with a Lie algebra structure. In fact, for $D \in  Der_{\a^k\b^l}(A)$  and $D' \in  Der_{\a^s\b^t}(A)$, we have $[D,D'] \in Der_{\a^{k+s}\b^{l+t}}(A)$, where $[D,D']$  is the standard commutator defined by $[D,D']= DD'-D'D$.

It is well known that  the BiHom-commutator  algebra $A^-$ of $A$ is a BiHom-Malcev algebra and
 the BiHom-anti-commutator  algebra $A^+$ of $A$ is a BiHom-Jordan algebra. We state the following result.

\begin{prop}
Let $D$ be an $\a^k\b^l$-derivation on $A$ then $D$ is still an $\a^k\b^l$-derivation on its associated BiHom-Jordan algebra $A^+$ and on its associated BiHom-Malcev algebra $A^-$.
\end{prop}

\begin{df}
A linear map $D\in End(A)$ is said to be an $\alpha^{k}\beta^{l}$-quasi-derivation of $A$ if there exists a linear map $D' \in End(A)$ such that
\begin{eqnarray*}
& [D,\alpha] = 0,~~[D^{'},\alpha]=0, [D,\beta]=0,~~[D^{'},\beta]=0,\\
& D^{'}(xy)= D(x)\alpha^{k}\beta^{l}(y)+\alpha^{k}\beta^{l}(x)D(y),~~\forall~~ x,y \in A,
\end{eqnarray*}
and we say that $D$ associates with $D'$.
\end{df}
\begin{df}
A linear map $D\in End(A)$ is said to be an $\alpha^{k}\beta^{l}$-generalized derivation  of $A$ if there exist linear maps $D', D'' \in End(A)$  such that
\begin{eqnarray*}
&& [D,\alpha] = 0,~~[D^{'},\alpha]=0,~~[D^{''},\alpha]=0,[D,\beta]=0,~~[D^{'},\beta]=0,~~[D^{''},\beta]=0,\\
&& D^{''}(xy)= D(x)\alpha^{k}\beta^{l}(y)+\alpha^{k}\beta^{l}(x)D^{'}(y),~~\forall x,y \in A,
\end{eqnarray*}
and we also say that $D$ associates with $D'$ and $D"$.
\end{df}

\begin{df}
An $\alpha^{k}\beta^{l}$-generalized derivation  $D$ of $A$ associated with $D'$ and $D"$ is said to be symmetric if (for any $x,y \in A$)
$$D^{''}(xy)= D(x)\alpha^{k}\beta^{l}(y)+\alpha^{k}\beta^{l}(x)D^{'}(y)=
D'(x)\alpha^{k}\beta^{l}(y)+\alpha^{k}\beta^{l}(x)D(y).
$$
\end{df}

The sets of generalized derivations and quasi-derivations and symmetric $\alpha^{k}\beta^{l}$-generalized derivations will be denoted by $GDer(A)$, $QDer(A)$ and $SGDer(A)$  respectively.

It is easy to see that
$$Der(A)\subset QDer(A)\subset SGDer(A)\subset GDer(A).$$
\begin{df}
A linear map $\theta \in End(A)$ is said to be an $\a^k\b^l$-centroid of $A$ if
\begin{align}
& [\theta,\alpha] = 0,\ [\theta,\b] = 0,\\
& \theta(xy)=\theta(x)\alpha^{k}\beta^{l}(y)=\alpha^{k}\beta^{l}(x)\theta(y),~~\forall~~ x,y \in A.
\end{align}
\end{df}
The  set of $\alpha^{k}\beta^l$-centroids of $A$ is denoted by $C(A)$.
\begin{df}
 A linear map $\theta \in End(A)$ is said to be an $\a^k\b^l$-quasi-centroid of $A$ if
\begin{align}
& [\theta,\alpha] = 0,\ [\theta,\b] = 0, \\
& \theta(x)\alpha^{k}\beta^{l}(y)= \alpha^{k}\beta^{l}(x)\theta(y),~~\forall~~ x,y \in A.
\end{align}
\end{df}
The set of $\alpha^{k}\beta^l$-quasi-centroids of $A$ is denoted by $QC(A)$.
It is obvious  that $C(A)\subset QC(A)$.
\begin{prop}
Let $D\in Der(A)$  and $\theta\in C(A)$. Then
\begin{eqnarray*}
&& [D,\theta]\in C(A).
\end{eqnarray*}
\end{prop}
\begin{proof} Assume that $D\in Der_{\alpha^{k}\beta^{l}},~~\theta \in C_{\alpha^{s}\beta^{t}}(A)$. For arbitrary $x,y \in A$, we have
\begin{align}
 D\theta (x)\alpha^{k+s}\beta^{l+t}(y)&=D(\theta (x)\alpha^{s}\beta^{t}(y))
-\alpha^{k}\beta^{l}(\theta (x))D(\alpha^{s}\beta^{t}(y))\nonumber\\
 &= D(\theta (x)\alpha^{s}\beta^{t}(y))
-\theta (\alpha^{k}\beta^{l}(x))D(\alpha^{s}\beta^{t}(y))\nonumber\\
&= D\theta (xy)-\alpha^{k+s}\beta^{l+t}(x)\theta D(y)\label{cent1}.
\end{align}
and
\begin{align}
 \theta D(x)\alpha^{k+s}\beta^{l+t}(y)&= \theta (D(x)\alpha^{k}\beta^{l}(y))\nonumber\\
 &= \theta D(xy)-\theta (\alpha^{k}\beta^{l}(x)D(y))\nonumber\\
&=  \theta D(xy)-(\alpha^{k+s}\beta^{l+t}(x)\theta D(y))\label{cent2}.
\end{align}
By making the difference of  equations
\eqref{cent1} and \eqref{cent2}, we get
\begin{align*}
   & [D,\theta] (xy)=[D,\theta] (x)\alpha^{k+s}\beta^{l+t}(y).
\end{align*}
\end{proof}
\begin{prop} $C(A)\subseteq QDer(A)$.
\end{prop}
\begin{proof}Let $\theta \in C_{\alpha^{k}\beta^{l}}(A)$ and $x,y \in A$, then we have
\begin{align*}
\theta(x)\alpha^{k}\beta^{l}(y)+ \alpha^{k}\beta^{l}(x)\theta(y)
&=\theta(x)\alpha^{k}\beta^{l}(y)+\theta(x)\alpha^{k}\beta^{l}(y)\\
&=2 \theta(xy)=D^{'}(xy).
\end{align*}
Then $\theta\in QDer_{\alpha^{k}\beta^{l}}(A)$.
\end{proof}

\begin{prop}$[QC(A),QC(A)]\subseteq QDer(A)$.
\end{prop}
\begin{proof} Assume that $\theta \in QC_{\alpha^{k}\beta^{l}}(A)$ and $\theta'\in QC_{\alpha^{s}\beta^{t}}(A) $. Then for all $x,y \in A$, we have
\begin{align*}
   & \theta(x)\alpha^{k}\beta^{l}(y)=\alpha^{k}\beta^{l}(x)\theta(y), \\
   & \theta'(x)\alpha^{s}\beta^{t}(y)=\alpha^{s}\beta^{t}(x)\theta'(y).
\end{align*}
Hence we obtain
\begin{eqnarray*}
[\theta,\theta'](x)\alpha^{k+s}\beta^{l+t}(y)&=
&(\theta\theta'-\theta'\theta)(x)\alpha^{k+s}\beta^{l+t}(y)\\
&=&\theta\theta'(x)\alpha^{k+s}\beta^{l+t}(y)-\theta'\theta(x)\alpha^{k+s}\beta^{l+t}(y)\\
&=&\alpha^{k+s}\beta^{l+t}(x)\theta'\theta(y)-\alpha^{k+s}\beta^{l+t}(x)\theta\theta'(y)\\
&=&-\alpha^{k+s}\beta^{l+t}(x)[\theta,\theta'](y),
\end{eqnarray*}
which implies that $[\theta,\theta'](x)\alpha^{k+s}\beta^{l+t}(y)+\alpha^{k+s}\beta^{l+t}(x)[\theta,\theta'](y)=0.$\\
Then $[\theta,\theta'] \in QDer_{\alpha^{k+s}\beta^{l+t}}(A)$ .
\end{proof}

\begin{prop}
The spaces $GDer(A),\  QDer(A)$ and $C(A)$ are BiHom-subalgebras of $(U,[\c,\c],\widetilde{\a},\widetilde{\b})$.
\end{prop}

\begin{proof}
For any generalized  $\a^k\b^l$-derivation $D$, it is obvious to see that $$D \circ \widetilde{\a}=\widetilde{\a} \circ D \in GDer_{\a^{k+1}\b^{l}}\ \textrm{and} \ D \circ \widetilde{\b}=\widetilde{\b} \circ D \in GDer_{\a^{k}\b^{l+1}}.$$ So that $\widetilde{\a} (GDer(A))\subseteq GDer(A)$ and $\widetilde{\b} (GDer(A))\subseteq GDer(A)$.

Let $D_1\in GDer_{\a^{k}\b^{l}}(A)$ and $D_2\in GDer_{\a^{s}\b^{t}}(A)$, then, for $x,y\in A$, we have
\begin{align}\label{GDer1}
  D_1"(xy) &= D_1(x)\a^{k}\b^{l}(y)+\a^{k}\b^{l}(x)D'_1(y) ,\\
  \label{GDer2}
  D_2"(xy) &= D_2(x)\a^{s}\b^{t}(y)+\a^{s}\b^{t}(x)D'_2(y).
\end{align}
Moreover
\begin{align*}
  [D_1,D_2](x) \a^{k+s}\b^{l+t}(y)  =&D_1D_2(x) \a^{k+s}\b^{l+t}(y) -D_2D_1(x) \a^{k+s}\b^{l+t}(y) \\
 = &D_1"(D_2(x)\a^{s}\b^{t}(y))-\a^{k}\b^{l}D_2(x)D_1'\a^{s}\b^{t}(y) \\
    - &D_2"(D_1(x)\a^{k}\b^{l}(y))+\a^{s}\b^{t}D_1(x)D_2'\a^{k}\b^{l}(y)\\
   =& D_1"D_2"(xy)-D_1"(\a^{s}\b^{t}(x)D_2'(y))\\
   -&D_2"(\a^{k}\b^{l}(x)D_1'(y))+\a^{k+s}\b^{l+t}(x)D_2 'D_1'(y)\\
     -&D_2"D_1"(xy)+D_2"(\a^{k}\b^{l}(x)D_1'(y))\\
   +&D_1"(\a^{s}\b^{t}(x)D_2'(y))-\a^{k+s}\b^{l+t}(x)D_1 'D_2'(y)\\
   =& [D_1",D_2"](xy)-\a^{k+s}\b^{l+t}(x)[D_1 ' ,D_2 '](y).
\end{align*}

 Then $[D_1,D_2]$ is a generalized  $\a^{k+s}\b^{l+t}$-derivation on $A$. Therefore $GDer(A)$  is a  BiHom-subalgebra of $(U,[\c,\c],\widetilde{\a},\widetilde{\b})$.

 Similarly, we can show that $ QDer(A)$ and $C(A)$ are also  BiHom-subalgebras of $(U,[\c,\c],\widetilde{\a},\widetilde{\b})$.
\end{proof}
\begin{prop}
We have $SGDer(A)=QDer(A)+QC(A)$.
\end{prop}

\begin{proof}
Let $D \in SGDer(A)$ associated with $D'$ and $D"$. Then for any $x, y \in A$, we have
$$D^{''}(xy)= D(x)\alpha^{k}\beta^{l}(y)+\alpha^{k}\beta^{l}(x)D^{'}(y)=
D'(x)\alpha^{k}\beta^{l}(y)+\alpha^{k}\beta^{l}(x)D(y).
$$
We remark that $D= \frac{D+D'}{2}+  \frac{D-D'}{2}$. So we will prove that
$ \frac{D+D'}{2} \in QDer(A)$ and $ \frac{D-D'}{2} \in CD(A)$. For this, take $x,y \in A$, then

\begin{align*}
 & \frac{D+D'}{2}(x)\a^k\b^l(y)+\a^k\b^l(x)\frac{D+D'}{2}(y)  \\
  &=\frac{1}{2}(D(x) \a^k\b^l(y)+ D'(x) \a^k\b^l(y)+
 \a^k\b^l(x)D(y)  +  \a^k\b^l(x)D'(y) ) \\
   & =D"(xy),
\end{align*}
which implies that  $\frac{D+D'}{2} \in QDer(A)$.

On the other hand, we have
\begin{align*}
    \frac{D-D'}{2}(x)\a^k\b^l(y) =&\frac{1}{2}(D(x)\a^k\b^l(y)-D'(x)\a^k\b^l(y) )\\
    =&\frac{1}{2}(\a^k\b^l(x)D(y)-\a^k\b^l(x)D'(y)) \\
    =&\a^k\b^l(x)\frac{D-D'}{2}(y),
\end{align*}
which means that   $\frac{D-D'}{2} \in CD(A)$. Hence $D \in QDer(A)+QC(A)$.
Moreover, it is straightforward  to prove the $QDer(A)+QC(A) \subset SGD(A)$.  Therefore $SGD(A)=QDer(A)+QC(A)$.

\end{proof}




\begin{thebibliography}{AA}

\bibitem{am1} H. Ataguema and A. Makhlouf, Deformations of ternary algebras, J. Gen. Lie Theory Appl., 1 (2007), 41--55.
\bibitem{am2} H. Ataguema and A. Makhlouf, Notes on cohomologies of ternary algebras of associative type, J. Gen. Lie Theory Appl., 3 (2009), 157--174.

\bibitem{Bordemann} M. Bordemann, Nondegenerate invariant bilinear forms on nonassociative algebras, Acta Math. Univ. Comenianae, Vol. LXVI(1), 1997, p. 151--201.

\bibitem{cce1} J. S. Carter, A. S. Crans, M. Elhamdadi, and M. Saito, Cohomology of categorical self-distributivity, J. Homotopy Relat. Struct., 3 (2008), 13--63.

\bibitem{cce2} J. S. Carter, A. S. Crans, M. Elhamdadi, and M. Saito, Cohomology of the adjoint of Hopf algebras, J. Gen. Lie Theory Appl., 2 (2008), 19--34.

\bibitem{ChtiouiMabroukMakhlouf} T. Chtioui, S. Mabrouk, A. Makhlouf, BiHom-alternative, BiHom-Malcev and BiHom-Jordan algebras,  Rocky Mountain Journal of Mathematics, 50(1), 69--90.
\bibitem{BiHomPreAlt} T. Chtioui, S. Mabrouk, A. Makhlouf, BiHom-pre-alternative algebras and BiHom-alternative quadri-algebras,  Bull. Math. Soc. Sci. Math. Roumanie. 63 (111), No. 1, 2020, 3--21


\bibitem{cg} S. Caenepeel, I. Goyvaerts, Monoidal Hom-Hopf algebras, Comm. Algebra 39 (2011), 2216--2240.

    \bibitem{ApurbaDas} A. Das, Cohomology of BiHom-associative algebras, Journal of Algebra and Its Applications,21 01  (2022). 

\bibitem{em1} M. Elhamdadi and A. Makhlouf, Cohomology and Formal deformations of Alternative algebras, J. G.  Lie Theory App. 5 (2011). 

\bibitem{em2}  M. Elhamdadi and A. Makhlouf,  Deformations of Hom-Alternative and Hom-Malcev algebras, Algebras Groups Geom. 28 (2011), no. 2, 117--145. 

\bibitem{f} A. Fialowski, Deformation of Lie algebras, Math. USSR Sb., 55 (1986), 467–473.

\bibitem{Gerstenhaber1} M. Gerstenhaber, On the deformation of rings and algebras. Ann. of Math. (2) 79 (1964), 59--103.
\bibitem{Gerstenhaber2} M. Gerstenhaber, On the deformation of rings and algebras. II. Ann. of Math. 84 (1966), 1--19.
\bibitem{Gerstenhaber3} M. Gerstenhaber, On the deformation of rings and algebras. III. Ann. of Math. (2) 88 (1968), 1--34.
\bibitem{Gerstenhaber4} M. Gerstenhaber, On the deformation of rings and algebras. IV. Ann. of Math. (2) 99 (1974), 257--276.

\bibitem{gmmp} G. Graziani, A. Makhlouf, C. Menini, F. Panaite, BiHom-associative algebras, BiHom-Lie algebras and BiHom-bialgebras, Symmetry Integrability Geom. Methods Appl. (SIGMA) 11 (2015), 086, 34 pages.

\bibitem{gm1} M. Goze and A. Makhlouf, On the rigid complex associative algebras, Comm. Algebra, 18 (1990), 4031--4046.
\bibitem{gm2} M. Goze and A. Makhlouf, Classification of rigid associative algebras in low dimensions, in Lois d'alg\`ebres et vari\'et\'es alg\'ebriques
(Colmar, 1991), vol. 50 of Travaux en Cours, Hermann, Paris, 1996, 5--22.

\bibitem{gr} M. Goze and E. Remm, Valued deformations of algebras, J. Algebra Appl., 3 (2004), 345--365.

\bibitem{hls}
J.T. Hartwig, D. Larsson, and S.D. Silvestrov, Deformations of Lie algebras using $\sigma$-derivations, J. Algebra 295 (2006) 314--361.


\bibitem{lmmp} ] L. Liu, A. Makhlouf, C. Menini, F. Panaite, BiHom-pre-Lie algebras, BiHom-Leibniz algebras and Rota-Baxter operators on BiHom-Lie algebras, Georgian Math. J. 28 , no. 4 (2021), 581--594.  

\bibitem{mak}
A. Makhlouf, Hom-alternative algebras and Hom-Jordan algebras, International Electronic Journal of Algebra, Vol. 8 (2010) 177--190. 

\bibitem{mak1} A. Makhlouf, Degeneration, rigidity and irreducible components of Hopf algebras, Algebra Colloq., 12 (2005), 241--254.

\bibitem{mak2} A. Makhlouf, A comparison of deformations and geometric study of varieties of associative algebras, Int. J. Math. Math. Sci., (2007), Art. ID 18915, 24.

\bibitem{ms}
A. Makhlouf and S. Silvestrov, Hom-algebra structures, J. Gen. Lie Theory Appl. 2 (2008) 51--64.

\bibitem{mst} M. Markl and J. D. Stasheff, Deformation theory via deviations, J. Algebra, 170 (1994), 122--155.

\bibitem{yau}
D. Yau, Hom-Novikov algebras, Journal of Physics A: Mathematical and theoretical 44.8 (2011). 

\end{thebibliography}
\end{document}